\documentclass[11pt]{amsart}
\usepackage{amscd}
\usepackage{comment}
\usepackage{amsmath, amssymb, amsgen, amsthm, amscd, xspace, color}
\newcommand{\pf}{\begin{proof}}
\newcommand{\epf}{\end{proof}}
\newcommand{\eq}{\begin{equation}}
\newcommand{\eeq}{\end{equation}}
\newcommand{\eqn}{\begin{equation*}}
\newcommand{\eeqn}{\end{equation*}}

\newtheorem{theorem}[equation]{Theorem}

\newtheorem{prop}[equation]{Proposition}
\newtheorem{lemma}[equation]{Lemma}

\theoremstyle{remark}

\theoremstyle{definition}

\numberwithin{equation}{section}
\allowdisplaybreaks[1]
\begin{document}

\title{SPECTRA FOR GELFAND PAIRS ASSOCIATED WITH THE FREE TWO STEP
NILPOTENT LIE GROUP}
\author{Jingzhe Xu}
\address[Xu]{Department of Mathematics, Hong Kong University of Science
and technology,
Clear Water Bay, Kowloon, Hong Kong SAR, China}
\email{jxuad@ust.hk}
\abstract{Let $F(n)$ be a connected and simply connected free 2-step
nilpotent lie group and $K$ be a compact subgroup of Aut($F(n)$). We say
that $(K,F(n))$ is a Gelfand pair when the set of integrable $K$-invariant
functions on $F(n)$ forms an abelian algebra under convolution. In this
paper, we consider the case when $K=O(n)$. In [1], we know the only possible Galfand pairs
for $(K,F(n))$ is $(O(n),F(n))$, $(SO(n),F(n))$. So we just consider the
case $(O(n),F(n))$, the other case can be obtained in the similar way.We study the natural topology
on $\Delta (O(n),F(n))$ given by uniform convergence on compact subsets in
$F(n)$. We show $\Delta (O(n),F(n))$ is a complete metric space. Our main
result gives a necessary and sufficient result for the sequence of the
"type 1" bounded $O(n)$-spherical functions uniform convergence to the
"type 1" bounded $O(n)$-spherical function on compact sets in $F(n)$.
What's more, the "type 1" bounded $O(n)$-spherical functions are dense in
$\Delta (O(n),F(n))$. Further, we define the Fourier transform according to the
"type 2" bounded $O(n)$-spherical functions and gives some basic
properties of it.}
\endabstract

%
\maketitle     
%
\section{Introduction}
Given a locally compact group $G$ and compact subgroup $K\subseteq G$, the
pair $(G,K)$ is called a Gelfand pair if $L^{1}(G//K)$, the space of
integrable, $K$-bi-invariant functions on $G$, is commutative. Perhaps the
best known examples are those defining symmetric spaces, that is, when $G$
is a connected semisimple lie group of finite center, and $K$ is a maximal
compact subgroup. The analysis associated with such pairs plays an
important role in the representation theory of semisimple lie groups and
has been extensively developed in the last four
decades.(cf.e.g.[2],[3]).In sharp contrast to this case, one might begin
by assuming that $G$ is a solvable lie group. But then, if $G$ is simply
connected for example, there maybe no non-trivial compact subgroups. One
can, however, consider pairs of the form $(K\rtimes G, K)$, where $K$ is a
compact subgroup of Aut($G$), the group of automorphism of $G$.

We study the connected, simply connected free two-step nilpotent lie
groups $F(n)$ for two reasons. Firstly, for any nilpotent lie group $N$
and a compact group $K\in Aut(N)$, $(K\rtimes N, K)$ is a Gelfand pair if
and only if $N$ is at most two step. Secondly, every two-step nilpotent
lie group is a quotient of a free two step nilpotent lie group. Therefore,
the connected, simply connected free two-step nilpotent lie group $F(n)$
is defined in Section 2.  We
construct an isomorphism between the heisenberg group and some group with
respect to $F(n)$. Then we get the "type 1" and "type 2" bounded
$O(n)$-spherical functions. Note that for the 2-step nilpotent lie group
$N$, and the Gelfand pair $(K,N)$, the bounded $K$-spherical functions
are the same as the positive definite $K$-spherical functions[4]. In fact,
this is not true in general for the semisimple case.

Our focus here is on the topology of the bounded "type 1", "type 2"
$O(n)$-spherical functions respectively, where the usual weak*-topology
coincides with the compact-open topology on $\Delta (O(n),F(n))$. Our main
result is stated as Theorem 4.1. It asserts that a "type 1" bounded
$O(n)$-spherical sequence $(\psi_{N})_{N}^{\infty}$ converges to a "type
1" bounded $O(n)$-spherical funtion $\psi$ if and only if
$\widehat{L_{\gamma_{j}}}(\omega_{\lambda_{N},\alpha_{N}})\rightarrow
\widehat{L_{\gamma_{j}}}(\omega_{\lambda,\alpha})$, where $i=1,\cdots
d$;$\widehat{T}(\omega_{\lambda_{N},\alpha_{N}})\rightarrow
\widehat{T}(\omega_{\lambda,\alpha})$;$\gamma_{N}\rightarrow \gamma$.Here
${L_{\gamma_{1}},\cdots L_{\gamma_{d}},T}$ is a generator of the algebra
$D_{k}(H_{n})$, which means the left-$H_{n}$-invariant and $K$-invariant
differential operators on $H(n)$. And $\omega_{\lambda_{N},\alpha_{N}}$,
$\omega_{\lambda,\alpha}$ are the "type 1"bounded $K$ spherical functions
on $H(n)$.$\gamma_{N},\gamma$ are the parameters that will be introduced
later. We require a careful analysis of the behavior of such eigenvalues,
and these results are described in Section 3.

We refine our description of the topology on the Gelfand space by proving
two final results. Theorem 4.5 asserts that $\Delta (K,F(n))$ is complete.
That is, if a sequence of bounded $K$-spherical functions converge to some
function in the compact-open topology, then the limit is necessary a
bounded $K$-spherical function. Later, we will asserts the the "type 1"
bounded $O(n)$-spherical functions are dense in $\Delta (O(n),F(n))$.
Section 5 contains a description of the Godemental-Plancherel measure and
the $O(n)$-spherical transform and then gives a definition for the Fouier
transform induced by the
"type 2" bounded $O(n)$-spherical functions. Also, we obtain some
important properties about it.

%
\section{Notation and Preliminaries}
In this section, we will introduce some basic knowledge and significant
results about the free two-step nilpotent lie group.

First Definition. Let $\mathcal{N}_{p}$ be the (unique up to isomorphism)
free two-step nilpotent Lie algebra with p generators. The definition
using the universal property of the free nilpotent Lie algebra can be
found in [7, Chapter V §5]. Roughly speaking, $\mathcal{N}_{p}$ is a
(nilpotent)Lie algebra with $p$ generators $X_{1},\cdots X_{p}$, such that
the vectors $X_{1},\cdots X_{p}$ and $X_{i,j}=[X_{i},X_{j}],i<j$ form a
basis; we call this basis the canonical basis of $\mathcal{N}_{p}$.

We denote by $\mathcal{V}$ and $\mathcal{Z}$, the vectors spaces generated
by the families of vectors $X_{1},\cdots X_{p}$ and
$X_{i,j}=[X_{i},X_{j}], 1\leqslant i<j \leqslant p$ respectively; these
families become the canonical base of $\mathcal{V}$ and $\mathcal{Z}$.
Thus $\mathcal{N}_{p}=\mathcal{V}\bigoplus \mathcal{Z}$, and $\mathcal{Z}$
is the center of $\mathcal{N}_{p}$. With the canonical basis, the vector
space $\mathcal{Z}$ can be identified with the vector space of
antisymmetric $p\times p$-matrices $\mathcal{A}_{p}$. Let
$z=dimZ=p(p-1)/2$.

The connected simply connected nilpotent Lie group which corresponds to
$\mathcal{N}_{p}$ is called the free two-step nilpotent Lie group and is
denoted $N_{p}$. We denote by exp:$\mathcal{N}_{p}\rightarrow N_{p}$the
exponential map.

In the following, we use the notations $X+A\in \mathcal{N}$,$exp(X+A)\in
N$ when $X\in \mathcal{V},A\in \mathcal{Z}$. We write $p=2p^{'} or
2p^{'}+1$.

A Realization of $\mathcal{N}_{p}$. We now present here a realization of
$\mathcal{N}_{p}$, which will be helpful
to define more naturally the action of the orthogonal group and
representations of $N_{p}$.

Let $(\mathcal{V},<,>)$ be an Euclidean space with dimension $p$. Let
$O(\mathcal{V})$ be the group of
orthogonal transformations of $\mathcal{V}$, and $SO(\mathcal{V})$ Its
special subgroup. Their common
Lie algebra denoted by $\mathcal{Z}$, is identified with the vector space
of antisymmetric transformations of $\mathcal{V}$. Let
$\mathcal{N}=\mathcal{V}\bigoplus \mathcal{Z}$ be the exterior direct sum
of the vector spaces $\mathcal{V}$ and $\mathcal{Z}$.

Let $[,]:\mathcal{V}\times \mathcal{V}\rightarrow \mathcal{Z}$ be the
bilinear application given by:

$[X,Y].(V)=<X,V>Y-<Y,V>X   \\ where X,Y,V\in \mathcal{V}$

We also denote by $[,]$ the bilinear application extended to
$\mathcal{N}\times \mathcal{N}\rightarrow \mathcal{N}$ by:

$[.,.]_{\mathcal{N}\times \mathcal{Z}}=[.,.]_{\mathcal{Z}\times \mathcal{N}}=0$
This application is a Lie bracket. It endows $\mathcal{V}$ with the
structure of a two-step nilpotent Lie algebra.

As the elements $[X,Y],X,Y\in \mathcal{V}$ generate the vector space
$\mathcal{Z}$, we also
define a scalar product $<，>$ on $\mathcal{Z}$ by:

$<[X,Y],[X^{'},Y^{'}]>=<X,X^{'}><Y,Y^{'}>-<X,Y^{'}><X^{'},Y>$

where $X,Y,X^{'},Y^{'}\in \mathcal{V}$.

It is easy to see $\mathcal{V}$ as a realization of $\mathcal{N}_{p}$ when
an orthonormal basis
$X_{1},\cdots X_{p}$ of $(\mathcal{V},<,>)$ is fixed.

We remark that $<[X,Y],[X^{'},Y^{'}]>=<[X,Y]X^{'},Y^{'}>$, and so we have
for an antisymmetric transformation $A\in \mathcal{Z}$, and for $X,Y\in
\mathcal{V}$:

$<A,[X,Y]>=<A.X,Y>$

This equality can also be proved directly using the canonical basis of
$\mathcal{N}_{p}$.

Actions of Orthogonal Groups. We denote by $O(\mathcal{V})$ the group of
orthogonal linear maps of $(\mathcal{V},<,>)$, and by $O_{p}$ the group of
orthogonal $p\times p$-matrices.

On $\mathcal{N}_{p}$ and $N_{p}$. The group $O(\mathcal{V})$ acts on the
one hand by automorphism on $\mathcal{V}$, on the other hand by the
adjoint representation $Ad_{\mathcal{Z}}$ on $\mathcal{Z}$. We obtain an
action of $O(\mathcal{V})$ on $\mathcal{N}=\mathcal{V}\bigoplus
\mathcal{Z}$. Let us prove that this action respects the Lie bracket of
$\mathcal{N}$. It suffices to show for
$X,Y,Z\in \mathcal{V}$ and $k\in O(\mathcal{V})$:
\begin{equation}\label{l-invariant elements}
\begin{split}
&[k.X,k.Y](V)=<k.X,V>k.Y-<k.Y,V>k.X\\
&k.(<X,k^{t}.V>Y-<Y,k^{t}.V>X\\
&=k.[X,Y](k^{-1}.V)=Ad_{\mathcal{Z}}k.[X,Y].
\end{split}
\end{equation}

We then obtain that the group $O(\mathcal{V})$ and also its special
subgroup $SO(\mathcal{V})$.
acts by automorphism on the Lie algebra $\mathcal{N}$, and finally on the
Lie group $N$.

Suppose an orthonormal basis $X_{1},\cdots X_{p}$ of $(\mathcal{V},<,>)$
is fixed; then the vectors
$X_{i,j}=[X_{i},X_{j}],1\leq i<j\leq p$, form an orthonormal basis of
$\mathcal{V}$ and we can identify:

the vector space $\mathcal{Z}$ and $\mathcal{A}_{p}$.

the group $O(\mathcal{V})$ with $O_{p}$.

the adjoint representation $Ad_{\mathcal{Z}}$ with the conjugate action of
$O_{p}$ and $\mathcal{A}_{p}: k.A=kAk^{-1}$, where $k\in O_{p}, A\in\mathcal{A}_{p}$.

Thus the group $O_{p}\sim O(\mathcal{V})$ acts on $\mathcal{V}\sim
\mathbb{R}^{p}$ and $\mathcal{Z}\sim \mathcal{A}_{p}$, and consequently on
$\mathcal{N}_{p}$. Those actions can be directly defined; and the equality
$[k.X,k.Y]=k.[X,Y],k\in O_{p},X,Y\in \mathcal{V}$, can then be computed.

On $\mathcal{A}_{p}$. Now we describe the orbits of the conjugate actions
of $O_{p}$ and $SO_{p}$ on $\mathcal{A}_{p}$. An arbitrary antisymmetric
matrix $A\in \mathcal{A}_{p}$ is $O_{p}$-conjugated to an antisymmetric
matrix $D_{2}(\wedge)$ where $\wedge=(\lambda_{1},\cdots
,\lambda_{p^{'}})\in \mathbb{R^{'}}$ and:

$D_{2}(\wedge)=\begin{bmatrix}
\ \lambda_{1}J &  0   &  0   & 0      \\
0 & \ddots  & 0  &  0   \\
0      &  0   &\lambda_{p^{'}}J  &   0       \\
0      &  0    &   0    &   (0)
\end{bmatrix}$

where $J:=\begin{bmatrix}
0      & 1      \\
-1      &  0
\end{bmatrix}$

((0) means that a zero appears only in the case $p=2p^{'}+1$) Furthermore,
we can assume that $\wedge$ is in $\overline{\mathcal{L}}$, where we
denote by $\mathcal{L}$ the set of $\wedge=(\lambda_{1},\cdots
,\lambda_{p^{'}})\in \mathbb{R^{'}}$ such that $\lambda_{1}\geq \cdots
\lambda_{p^{'}}\geq 0$.

Parameters. To each $\wedge\in \overline{\mathcal{L}}$, we
associate:$p_{0}$ the number of $\lambda_{i}\neq0$, $p_{1}$ the number of
distinct $\lambda_{i}\neq0$, and $\mu_{1},\cdots \mu_{p^{1}}$ such that:

$\{\mu_{1}> \mu_{2}> \cdots >\mu_{p_{1}}>0\}=\{\lambda_{1}\geq
\lambda_{2}\geq \cdots \geq \lambda_{p_{0}}>0\}$

We denote by $m_{j}$ the number of $\lambda_{i}$ such that
$\lambda_{i}=\mu_{j}$, and we put

$m_{0}:=m_{0}^{'}:=0$ and for $j=1,\cdots p_{1} \
m_{j}^{'}:=m_{1}+\cdots+m_{j}$.

For $j=1,\cdots p_{1}$, let $pr_{j}$ be the orthogonal projection of
$\mathcal{V}$ onto the space generated by the vectors $X_{2i-1},X_{2i}$,
for $i=m_{j-1}^{'}+1,\cdots m_{j}^{'}$.

Let $\mathcal{M}$ be the set of $(r,\wedge)$ where $\wedge \in
\mathcal{L}$, and $r\geq 0$, such that $r=0$ if $2p_{0}=p$.

Expression of the bounded spherical functions. The bounded spherical
functions of $(N_{p},K)$ for $K=O_{p}$, are parameterized by

$(r,\wedge)\in \mathcal{M}$ (with the previous notations
$p_{0},p_{1},\mu_{i},pr_{j}$ associated to $\wedge$),

$l\in \mathbb{N}^{p_{1}}$ if $\wedge\neq 0$, otherwise $\varnothing$.

Let $(r,\wedge)$,$l$ be such parameters. Then we have the following two
types of bounded $O(n)$-spherical functions:

For $n=exp(X+A)\in N$.

Type 1:$\phi^{r,\wedge,l}(n)=\int_{K}e^{ir<X_{p}^{*},k.X
>}\omega_{\wedge,l}(\Psi_{2}^{-1}(\overline{q_{1}}(k.n)))dk$.

Type 2:$\phi^{\upsilon}(n)=\int_{K}e^{ir<X_{p}^{*},k.X >}dk$.
Here $X_{p}^{*}$ is the unit $K_{\rho}$-fixed invariant vector.
For a Gelfand pair $(H^{p_{0}},K(m;p_{1};p_{0}))$, we have
$\omega_{\wedge,l}$ is the "type 1" bounded $K(m;p_{1};p_{0})$-spherical
functions for the Heisenberg group $H^{p_{0}}$. We will introduce it next.
$\Psi_{2}$ is an isomorphism between $H^{p_{0}}$ with a group with respect
to $F(n)$, which will be introduced later.

We use the following law of the Heisenberg group $\mathbb{H}^{p_{0}}$:

$\forall h=(z_{1},\ldots ,z_{p_{0}},t)$ ,
$h^{'}=(z_{1}^{'},\ldots ,z_{p_{0}}^{'},t^{'})\in \mathbb{H}^{p_{0}}=\mathbb{C}^{p_{0}}\times \mathbb{R}$

$h.h^{'}=(z_{1}+z_{1}^{'},\ldots ,z_{p_{0}}+z_{p_{0}}^{'},t+t^{'}+\frac{1}{2}\sum_{i=1}^{p_{0}}\mathfrak{F}z_{i}\overline{z}_{i}^{'})$

The unitary $p_{0}\times p_{0}$ matrix group $U_{p_{0}}$ acts by automorphisms on $\mathbb{H}^{p_{0}}$. Let us describe some subgroups of $U_{p_{0}}$. Let $p_{0},p_{1}\in \mathbb{N}$, and $m=(m_{1},\ldots ,m_{p_{1}})\in \mathbb{N}^{p_{1}}$ be fixed such that $\sum_{j=1}^{p_{1}}m_{j}=p_{0}$. Let $K(m;p_{1};p_{0})$ be the subgroup of $U_{p_{0}}$ given by:

$K(m;p_{1};p_{0})=U_{m_{1}}\times \ldots \times U_{m_{p_{1}}}$.

The expression of spherical functions of $(\mathbb{H}^{p_{0}},K(m;p_{1};p_{0}))$ can be found in the same way as in the case $m=(p_{0})$, $p_{1}=1$
i.e. $K=U_{p_{0}}$.

Stability group $K_{\rho}=\{k\in K: k.\rho=\rho\}=\{k\in K\subset G: k.f\in N.f\}$. The aim of this paragraph is to describe the stability group $K_{\rho}$
of $\rho\in T_{rX_{p}}^{*}+D_{2}(\wedge)$.

Before this, let us recall that the orthogonal $2n\times 2n$ matrices which commutes with $D_{2}(1,\ldots ,1)$ have determinant one and form the group$Sp_{n}\bigcap O_{2n}$. This group is isomorphism to $U_{n}$; the isomorphism is denoted $\psi_{1}^{(n)}$, and satisfies:

$\forall k,X$: $\psi_{c}^{(n)}(k.X)=\psi_{1}^{(n)}(K)\psi_{c}^{(n)}(X)$,

where $\psi_{c}^{(n)}$ is the complexification :

$\psi_{c}^{(n)}(x_{1},y_{1};\ldots ;x_{n},y_{n})=(x_{1}+iy_{1},\ldots ,x_{n}+iy_{n})$.

Now, we can describe $K_{\rho}$:
\begin{prop}
Let $(r,\Lambda)\in \mathcal{M}$. Let $p_{0}$ be the number of $\lambda_{i}\neq 0$, where $\wedge=(\lambda_{1},\ldots ,\lambda_{p^{'}})$, and $p_{1}$ the number of distinct $\lambda_{i}\neq 0$. We set $\widetilde{\wedge}=(\lambda_{1},\ldots ,\lambda_{p_{0}})\in \mathbb{R}^{p_{0}}$.
Let $\rho\in T_{f}$ where $f=rX_{p}^{*}+D_{2}(\wedge)$.

If $\wedge=0$, then $K_{\rho}$ is the subgroup of $K$ such that $k.rX_{p}^{*}=rX_{p}^{*}$ for all $k\in K_{\rho}$.

If $\wedge\neq 0$, then $K_{\rho}$ is the direct product $K_{1}\times K_{2}$, where:

$K_{1}=\{k_{1}=\begin{bmatrix}
\widetilde{k_{1}}     &      0      \\
0 & Id
\end{bmatrix} \mid \widetilde{k_{1}}\in SO(2p_{0}) \  D_{2}(\widetilde{\wedge})\widetilde{k}_{1}=\widetilde{k}_{1}D_{2}(\widetilde{\wedge})\}$

$K_{2}=\{k_{2}=\begin{bmatrix}
Id    &      0      \\
0 & \widetilde{k_{1}}
\end{bmatrix} \mid \widetilde{k_{2}}.rX_{p}^{*}=rX_{p}^{*}\}$.

Furthermore, $K_{1}$ is isomorphism to the group $K(m;p_{0};p_{1})$.
\begin{proof}
We keep the notations of this proposition, and we set $A^{*}=D_{2}(\wedge)$ and $X^{*}=rX_{p}^{*}$. It is easy to prove:

$K_{\rho}=\{k\in K: kA^{*}=A^{*}k \ and \ kX^{*}=X^{*}k\}$.

If $\wedge=0$, since $K_{\rho}$ is the stability group in $K$ of $X^{*}\in \mathcal{V}^{*}\sim \mathbb{R}^{p}$. So the
first part of Proposition 2.2 is proved.

Let us consider the second part. $\wedge\neq 0$ so we have

$A^{*}=\begin{bmatrix}
D_{2}(\widetilde{\wedge})    &     0      \\
0      &  0
\end{bmatrix}$ \ with \ $D_{2}(\widetilde{\wedge})=\begin{bmatrix}
\mu_{1}J_{m_{1}}     & 0 & 0      \\
0 & \ddots & 0 \\
0      & 0  & \mu_{p_{1}}J_{m_{p_{1}}}
\end{bmatrix}$

Let $k\in K_{\rho}$. From above computation, the matrices $k$ and $A^{*}$ commute and we have:

$k=\begin{bmatrix}
\widetilde{k}_{1}   &    0      \\
0      & \widetilde{k}_{2}
\end{bmatrix}$ \ with \ $\widetilde{k}_{1}\in O(2p_{0})$ \ and \ $\widetilde{k}_{2}\in O(p-2p_{0})$

furthermore, $\widetilde{k}_{2}.X^{*}=X^{*}$, and the matrices $\widetilde{k}_{1}$ and $D_{2}(\widetilde{\wedge}^{*})$ commute. So
$\widetilde{k}_{1}$ is the diagonal block matrix, with block $[\widetilde{k}_{1}]_{j}\in O(m_{j})$ for $i=1, \ldots ,p_{1}$. Each block
$[\widetilde{k}_{1}]_{j}\in O(m_{j})$ commutes with $J_{m_{j}}$. So on one hand, we have det$[\widetilde{k}_{1}]_{j}=1$, det$\widetilde{k}_{1}=1$,
and one the other hand, $[\widetilde{k}_{1}]_{j}\in O(m_{j})$ corresponds to a unitary matrix $\psi_{1}^{(m_{j})}([\widetilde{k}_{1}]_{j})$. Now
we set for $k_{1}\in K_{1}$:

$\Psi_{1}(k_{1})=(\psi_{1}^{(m_{j})}([\widetilde{k}_{1}]_{1}),\ldots ,\psi_{1}^{(m_{j})}([\widetilde{k}_{1}]_{p_{1}}))$

$\Psi_{1}:K_{1}\longrightarrow K(m;p_{0};p_{1})$ is a group isomorphism. 
\end{proof}
\end{prop}

Quotient group $\overline{N}=N/ker\rho$. In this paragraph, we describe
the quotient groups $N/ker\rho$ and $G/ker\rho$, for some $\rho\in
\widehat{N}$. This will permit in the next paragraph to reduce the
construction of the bounded spherical functions on $N_{p}$ to known
questions on Euclidean and Heisenberg groups. For a representation
$\rho\in \widehat{N}$, we will denote by:

$ker\rho$ the kernel of $\rho$.

$N/ker\rho$ its quotient group and $\overline{N}$ its lie algebra.

$(\mathcal{H},\overline{\rho})$ the induced representation on $\overline{N}$.

$\overline{n}\in \overline{N}$ and $\overline{Y}\in
\overline{\mathcal{N}}$ the image of $n\in N$ and $Y\in \mathcal{N}$
respectively by the canonical projections $N\rightarrow \overline{N}$ and
$\mathcal{N}\rightarrow \overline{\mathcal{N}}$

Now, with the help of the canonical basis, we choose

$E_{1}=\mathbb{R}X_{1}\bigoplus \cdots \bigoplus  \mathbb{R}X_{2p_{0}-1}$

as the maxiamal totally isotropic space for $\omega_{D_{2}(\wedge),r}$.
The quotient lie algebra $\overline{\mathcal{N}}$ has the natural
basis:           .
You can refer to [6].

Here, we have denoted
$\left |\wedge  \right |=(\sum_{j=1}^{p^{'}}\lambda_{j}^{2})^{\frac{1}{2}}=\left |D_{2}(\wedge)
\right |$ (for the Euclidean norm on $\mathcal{Z}$.

Let $\overline{\mathcal{N}_{1}}$ be the Lie sub-algebra of
$\overline{\mathcal{N}}$, with basis $\overline{X_{1}},\cdots
,\overline{X_{2p_{0}}},\overline{B}$, and $\overline{N_{1}}$ be its
corresponding connected simply connected nilpotent lie group. We define
the mapping : $\Psi_{2}:\mathbb{H}^{p_{0}}\rightarrow
\overline{\mathcal{N}_{1}}$ for $h=(x_{1}+iy_{1},\cdots
,x_{p_{0}}+iy_{p_{0}},t)\in \mathbb{H}^{p_{0}}$ by:

$\Psi_{2}(h)=exp(\sum_{j=1}^{p_{0}}\sqrt{\frac{\left |\wedge  \right
|}{\lambda_{j}}}(x_{j}\overline{X_{2j-1}}+y_{j}\overline{X_{2j}})+t\overline{B})$

We compute that each lie bracket of two vectors of this basis equals
zeros, except:

$[\overline{X_{2i-1}},\overline{X_{2i}}]=\frac{\lambda_{i}}{\left |\wedge
\right |}\overline{B},\ i=1, \cdots ,p_{0}$.

From this, it is easy to see:
\begin{theorem}\label{equal}
$\Psi_{2}$ is a group isomorphism between $\overline{N_{1}}$ and
$\mathbb{H}^{p_{0}}$
\end{theorem}

Finally, we note that

$\overline{q_{1}}:N\rightarrow \overline{N_{1}}$ is the canonical projection.

%
\section{some analysis on the heisenberg group $H_{n}$}
The whole parts of this section can be referred from [7].
A result due to Howe and Umeda (cf. [8]) shows that
$\mathbb{C}[v_{R}]^{K}$ is freely
generated as an algebra. So there are polynomials
$\gamma_{1},\cdots ,\gamma_{d}\in \mathbb{C}[v_{R}]^{K}$ so that
$\mathbb{C}[v_{R}]^{K}=\mathbb{C}[\gamma_{1},\cdots ,\gamma_{d}]$.

We call $\gamma_{1},\cdots ,\gamma_{d}$ the fundamental invariants.

Invariant different operators. The algebra $\mathbb{D}(H_{n})$ of
left-invariant
differential operators on $H_{n}$ is generated by
$\{Z_{1},\cdots ,Z_{n},\overline{Z_{1}},\cdots ,\overline{Z_{n}},T\}$. We
denote the subalgebra of $K$-invariant differential operators by

$\mathbb{D}_{K}(H_{n}):=\{D\in \mathbb{D}(H_{n})\mid D(f\circ k)=D(f)\circ
k \ for \ k\in K,f\in C^{\infty}(H_{n})\}$

From now on, we always suppose $(K,H_{n})$ is a Gelfand pair, and if this
is true, $\mathbb{D}_{K}(H_{n})$ is an abelian algebra.

We define $p(Z,\overline{Z})$ as follows:

$p(Z,\overline{Z}):=\sum c_{a,b}Z^{a}\overline{Z^{b}}=\sum
c_{a,b}Z_{1}^{a_{1}}\cdots Z_{n}^{a_{n}}\overline{Z_{1}}^{b_{1}}\cdots
\overline{Z_{n}}^{b_{n}}$

belongs to $\mathbb{D}_{k}(H_{n})$. Note that the operator $L_{p}$ is
intrinsically defined, whereas
$p(Z,\overline{Z})$ depends on the basis used to identify $V$ with
$\mathbb{C}^{n}$. One has

$L_{p}=Sym(p(Z,\overline{Z}))$,

where Sym is the linear map characterized by

$Sym(Z^{a}\overline{Z^{b}})=\frac{1}{(\left |a  \right |+\left |b  \right
|)!}\sum_{\sigma\in S_{(\left |a  \right |+\left |b  \right
|}}\sigma(Z^{a}\overline{Z}^{b})$.

Here, as usual, $\left |a  \right |=a_{1}+\cdots a_{n}$ and
$\sigma(Z^{a}\overline{Z^{b}})$ denotes the result of applying the
permutation $\sigma$ to the $\left |a  \right |+\left |b  \right |$ terms
in $Z^{a}\overline{Z}^{b}$.

For a d-multi-index $a$, let $\gamma^{a}:=\gamma_{1}^{a_{1}}\cdots
\gamma_{d}^{a_{d}}$, $\left \|a  \right \|:= a_{1}\left |\delta_{1}
\right|+\cdots +a_{d}\left |\delta_{d} \right|$ (the homogeneous degree of
$\gamma^{a}$), and $L_{\gamma}^{a}:=L_{\gamma_{1}}^{a_{1}}\cdots
L_{\gamma_{d}}^{a_{d}}$. Using the definition of the map
$\sim_{\mathcal{S}}$, together with the fact that
$[Z_{j},\overline{Z}_{j}]=-2iT$, one sees that

$L_{\gamma^{a}}=L_{\gamma}^{a}+\sum_{\left \|b  \right \|<\left \|a
\right \|}c_{a,b}L_{\gamma}^{b}T^{\left \|b  \right \|-\left \|a  \right
\|}$

for some coefficients $c_{a,b}\in \mathbb{C}$. Since $\gamma_{1},\cdots
,\gamma_{d}$ generates
$\mathbb{C}[v_{R}]^{K}$, it follows easily that
$\{L_{\gamma^{1}},\cdots ,L_{\gamma^{d}},T\}$ generates the algebra
$\mathbb{D}_{K}(H_{n})$.

Therefore, For any $H_{n}$-spherical function $\psi$.  $\psi$ is an
eigenfunction for every $D\in \mathbb{D}_{K}(H_{n})$ if and only if $\psi$
is an eigenfunction for each of $L_{\gamma^{1}},\cdots ,L_{\gamma^{d}},T$.
We write $\widehat{D}(\psi)$ for the eigenvalue of $D\in
\mathbb{D}_{K}(H_{n})$, that is $D(\psi)=\widehat{D}(\psi)\psi$. Note that
since $\psi(0,0)=1$, one has $\widehat{D}(\psi)=D(\psi)(0,0)$.
\begin{theorem}\label{equal}
The bounded $K$-spherical functions on $H_{n}$ are parametrized by the set
$(\mathbb{R}^{\times}\times \wedge)\cup (V/K)$ via

$\bigtriangleup(K,H_{n})=\{\phi_{\lambda,\alpha }\mid \lambda\in
\mathbb{R}^{\times},\alpha\in \wedge\}\cup \{\eta_{K_{\omega}}\mid
\omega\in V\}$

Note that, for $\psi\in \bigtriangleup(K,H_{n})$, one has
$\psi(z,t)=e^{i\lambda t}\psi(z,0)$,

where $\lambda=-i\widehat{T}(\psi)\in \mathbb{R}$.
\end{theorem}

\begin{lemma}\label{twist}
For $p\in \mathbb{C}[v_{R}]^{K}$ and $\psi\in \bigtriangleup(K,H_{n})$,
one has

$\widehat{L}_{p}(\psi)=\partial_{p}(\psi)(0,0)$,

where $\partial_{p}=p(2\frac{\partial}{\partial
\overline{z}},2\frac{\partial}{\partial z})$. That is $\partial_{p}$ is
the operator obtained by replacing each occurrence of $z_{j}$ in $p$ by
$2\frac{\partial}{\partial \overline{z_{j}}}$ and each $\overline{z_{j}}$
by $2\frac{\partial}{\partial z_{j}}$.
\end{lemma}

\begin{lemma}\label{twist}
$\widehat{L}_{p_{\alpha }}(\eta_{\omega})=(-1)^{\left |\alpha  \right
|}P_{\alpha}(\omega)$
\end{lemma}

\begin{lemma}\label{twist}
$\widehat{L}_{p_{\alpha }}(\phi_{\lambda, \beta})=\widehat{L}_{p_{\alpha
}}(\phi_{\beta}) \ for \ \alpha, \beta\in \wedge, \lambda\in
\mathbb{R}^{\times}$.
\end{lemma}

\begin{lemma}\label{twist}
$\widehat{L}_{p_{\alpha }}$ is a real number with sign $(-1)^{\left
|\alpha  \right |}$ for all $\alpha\in \wedge$ and $\psi \in
\bigtriangleup(K,H_{n})$.
\end{lemma}

\begin{lemma}\label{twist}
The eigenvalues for $L_{\gamma_{0}}$ on the $U(n)$-spherical functions of
type 1 are $\widehat{L}_{\gamma_{0}}(\phi_{\lambda,r})=-\left |\lambda
\right |(2r+n)$
\end{lemma}

\begin{lemma}\label{twist}
$\left |\widehat{L}_{p_{m}}(\phi_{r})  \right |\leq \begin{pmatrix}
n+r+m-1\\
m
\end{pmatrix}$

where $p_{m}$ is the $U(n)$-invariant polynomial obtained from $P_{m}(V)$
\end{lemma}

\begin{theorem}\label{equal}
$For \psi\in \bigtriangleup(K,H_{n})$, one has

$\psi(z,0)=\sum_{\delta\in
\wedge}\frac{\widehat{L}_{p_{\delta}}(\psi)}{dim(P_{\delta})}p_{\delta}(z)$,

where the series converges absolutely and uniformly on compact subsets in
$V$. Thus we have the following series expansions for the
$K$-spherical functions of types 1 and 2 respectively:

$\phi_{\lambda, \alpha}(z,t)=e^{i\lambda t}\sum_{\delta\in
\wedge}\frac{\left |\lambda  \right |^{\left |\delta  \right
|}\widehat{L}_{p_{\delta}}(\phi_{\delta})}{dim(P_{\delta})}p_{\delta}(z)$

$\eta_{w}(z,t)=\sum_{\delta\in \wedge}\frac{(-1)^{\left |\delta  \right
|}\widehat{L}_{p_{\delta}}(\omega)}{dim(P_{\delta})}p_{\delta}(z)$.

Here, convergence is absolute and uniform on compact subsets in $H_{n}$.
\end{theorem}
\begin{proof}
The expansions for $\phi_{\lambda, \alpha}(z,t)$ and $\eta_{w}(z,t)$
follow immediately
from that for $\psi(z,0)$ together with Lemmas 3.4 and 3.3. It is a general
fact that the spherical functions for a Gelfand pair $(G,K)$ are real
analytic
(cf.  Proposition 1.5.15 in [1]).  For pairs of the form $(K,H_{n})$, one can
see this directly from the functional forms of the two types of $K$-spherical
functions. Write the Taylor series expansion of $\psi(z,0)$ centered at
$z=0$ as $\psi(z,0)=\sum_{m=0}^{\infty}h_{m}(z)$, where $h_{m}(z)$ is a
homogeneous polynomial of degree $m$ on $V_{\mathbb{R}}$(i.e. in the
variables $(z,\overline{z})$). Since $\psi$ is $K$-invariant, one sees by
$K$- averaging this expression that each $h_{m}$ is $K$-invariant. As
$\{p_{\delta}\mid \delta\in \wedge\}$ is a basis for
$\mathbb{C}[V_{\mathbb{R}}]^{K}$, we can rewrite the Taylor series as

$\psi(z,0)=\sum_{\delta \in \wedge}c_{\delta}p_{\delta}(z)$

for some coefficients $c_{\delta}$. Note that since Taylor series converge
absolutely, it
is not necessary to specify an ordering on the set $\wedge$ of indices for
this sum.
We use Lemma 3.2 and perform term-wise differentiation of this Taylor
series to obtain
$\widehat{L}_{p_{\alpha }}(\psi)=\partial_{p_{\alpha
}}(\psi)(0,0)=\sum_{\delta\in
\wedge}c_{\delta}\partial_{p_{\alpha}}(p_{\delta})(0)$.

Let $\{v_{1},\cdots ,v_{dim(P_{\alpha})}\}$ be an orthonormal basis for
$P_{\alpha}$ and $\{u_{1},\cdots ,u_{dim(P_{\alpha})}\}$ be an orthonormal
basis for $P_{\delta}$. Thus $p_{\alpha}=\sum
v_{j}(z)\overline{v}_{j}(\overline{z})$, $p_{\delta}=\sum
u_{j}(z)\overline{u}_{j}(\overline{z})$ and

\begin{equation}\label{l-invariant elements}
\begin{split}
&\partial_{p_{\alpha}}(p_{\delta})(0)=\sum_{i,j}v_{j}(2\frac{\partial}{\partial
\overline{z}})\overline{u}_{i}(\overline{z})\overline{v}_{j}(2\frac{\partial}{\partial
z})u_{i}(z)\mid _{z=0}\\
&=\sum_{i,j}\left |<v_{j},u_{i} >_{\mathcal{F}}  \right |^{2}=
\begin{cases}
0,  & \mbox{if }\delta \neq \alpha \\
dim(P_{\alpha}), & \mbox{if }\delta =\alpha
\end{cases}.
\end{split}
\end{equation}

Hence $\widehat{L}_{p_{\alpha }}(\psi)=c_{\alpha}dim(P_{\alpha})$ and
$c_{\delta}=\widehat{L}_{p_{\alpha }}(\psi)/dim(P_{\delta})$. Since this
is a
Taylor series, the convergence is absolute and uniform on compact sets.
\end{proof}

\section{The proof of the main theorem}
\begin{theorem}\label{equal}
Let $(\psi_{N})_{N=1}^{\infty}$ be a sequence of bounded $O(n)$-spherical
functions of type 1, and $\psi$ is a bounded $O(n)$-spherical function of
type 1. Then $\psi_{N}$ converges to $\psi$ in the topology of
$\bigtriangleup(O(n),F(n))$(i.e.uniformly on compact sets) if and only if
$\widehat{L}_{\gamma_{i}}(\omega_{\lambda_{N},\alpha_{N}})\rightarrow
\widehat{L}_{\gamma_{i}}(\omega_{\lambda,\alpha})$ \ for \  $i=1,\cdots
d$; $\widehat{T}(\omega_{\lambda_{N},\alpha_{N}})\rightarrow
\widehat{T}(\omega_{\lambda,\alpha})$ and $r_{N}\rightarrow r$.
\end{theorem}
\begin{proof}
Suppose that $(\phi_{\gamma_{N},
\lambda_{N},\alpha_{N}}^{\nu}(n))_{N=1}^{\infty}$ converges uniformly to
$\phi_{\gamma, \lambda,\alpha}^{\nu}(n)$ on compact sets of $F(n)$.

Form [6], we can find a $K$-fixed vector $X\in
\mathcal{N}_{p}$(i.e.$k.X=X, \ \forall k\in K$). Let
$n_{t}=exp(tX)$,$-\infty <t< \infty$. From Theorem 3.8, we obtain

\begin{equation}\label{l-invariant elements}
\begin{split}
&\phi_{\gamma_{N},
\lambda_{N},\alpha_{N}}^{\nu}(n_{t})=\int_{K}e^{ir_{N}<X_{p}^{*},k.X
>t}\sum_{\delta \in
\wedge}\frac{\widehat{L}_{P_{\delta}}(\omega_{\lambda_{N},\alpha_{N}})}{dim(P_{\delta})}P_{\delta}(\Psi_{2}^{-1}(\overline{q_{1}(k.exp(tX))}))dk
\\
&=\sum_{\delta \in
\wedge}\frac{\widehat{L}_{P_{\delta}}(\omega_{\lambda_{N},\alpha_{N}})}{dim(P_{\delta})}e^{ir_{N}<X_{p}^{*},X>t}P_{\delta}(\Psi_{2}^{-1}(\overline{q_{1}(exp(X))}))t^{2\left
|\delta  \right |}
\end{split}
\end{equation}

Similarly, $\phi_{\gamma, \lambda,\alpha}^{\nu}(n_{t})=\sum_{\delta \in
\wedge}\frac{\widehat{L}_{P_{\delta}}(\omega_{\lambda,\alpha})}{dim(P_{\delta})}e^{ir<X_{p}^{*},X>t}P_{\delta}(\Psi_{2}^{-1}(\overline{q_{1}(exp(X))}))t^{2\left
|\delta  \right |}$

Since $\phi_{\gamma_{N}, \lambda_{N},\alpha_{N}}^{\nu}(n_{t})\rightarrow
\phi_{\gamma, \lambda,\alpha}^{\nu}(n_{t})$, We have

$\sum_{\delta \in
\wedge}\frac{\widehat{L}_{P_{\delta}}(\omega_{\lambda_{N},\alpha_{N}})}{dim(P_{\delta})}e^{ir_{N}<X_{p}^{*},X>t}P_{\delta}(\Psi_{2}^{-1}(\overline{q_{1}(exp(X))}))\rightarrow
\ \sum_{\delta \in
\wedge}\frac{\widehat{L}_{P_{\delta}}(\omega_{\lambda,\alpha})}{dim(P_{\delta})}e^{ir<X_{p}^{*},X>t}P_{\delta}(\Psi_{2}^{-1}(\overline{q_{1}(exp(X))}))$.

Differentiate with respect to $t$ respectively  we have

$ir_{N} \cdot <X_{p}^{*},X> \sum_{\delta \in
\wedge}\frac{\widehat{L}_{P_{\delta}}(\omega_{\lambda_{N},\alpha_{N}})}{dim(P_{\delta})}e^{ir_{N}<X_{p}^{*},X>t}P_{\delta}(\Psi_{2}^{-1}(\overline{q_{1}(exp(X))}))\rightarrow$

ir $\cdot <X_{p}^{*},X> \sum_{\delta \in
\wedge}\frac{\widehat{L}_{P_{\delta}}(\omega_{\lambda,\alpha})}{dim(P_{\delta})}e^{ir<X_{p}^{*},X>t}P_{\delta}(\Psi_{2}^{-1}(\overline{q_{1}(exp(X))}))$.

Therefore, $r_{N}\rightarrow r$.

Next, we choose a skew symmetric matrix $A$ such that

$\int_{K}\Psi_{2}^{-1}(\overline{q_{1}(k.exp(A))})dk\neq 0$. Since

$\omega_{\lambda_{N},
\alpha_{N}}(\Psi_{2}^{-1}(\overline{q_{1}(k.exp(tA))})=\omega_{\lambda_{N},
\alpha_{N}}(\Psi_{2}^{-1}(\overline{q_{1}(k.exp(A))}t)$

$=e^{\widehat{T}(\omega_{\lambda_{N},
\alpha_{N}})\Psi_{2}^{-1}(\overline{q_{1}(k.exp(A)))}t}$ and

$\omega_{\lambda,
\alpha}(\Psi_{2}^{-1}(\overline{q_{1}(k.exp(tA))})=e^{\widehat{T}(\omega_{\lambda,
\alpha})\Psi_{2}^{-1}(\overline{q_{1}(k.exp(A)))}t}$. Therefore,

$\phi_{\gamma_{N},
\lambda_{N},\alpha_{N}}^{\nu}(exptA)=\int_{K}\omega_{\lambda_{N},
\alpha_{N}}(\Psi_{2}^{-1}(\overline{q_{1}(k.exp(tA))})dk=$

$\int_{K}e^{\widehat{T}(\omega_{\lambda_{N},
\alpha_{N}})\Psi_{2}^{-1}(\overline{q_{1}(k.exp(A)))}t}dk$.

$\phi_{\gamma, \lambda, \alpha}^{\nu}(exptA)=\int_{K}\omega_{\lambda,
\alpha}(\Psi_{2}^{-1}(\overline{q_{1}(k.exp(tA))})dk=$

$\int_{K}e^{\widehat{T}(\omega_{\lambda,
\alpha})\Psi_{2}^{-1}(\overline{q_{1}(k.exp(A)))}t}dk$.

Since $\phi_{\gamma_{N}, \lambda_{N},\alpha_{N}}^{\nu}(exptA)$ converges to
$\phi_{\gamma, \lambda, \alpha}^{\nu}(exptA)$ uniformly for all $-\infty
<t<\infty$.

Differentiate with respect to t respectively and let $t=0$, we obtain

$\widehat{T}(\omega_{\lambda_{N},
\alpha_{N}})\int_{K}\Psi_{2}^{-1}(\overline{q_{1}(k.exp(A))})dk\rightarrow
\widehat{T}(\omega_{\lambda,
\alpha})\int_{K}\Psi_{2}^{-1}(\overline{q_{1}(k.exp(A))})dk$.

Therefore, $\widehat{T}(\omega_{\lambda_{N}, \alpha_{N}})\rightarrow
\widehat{T}(\omega_{\lambda, \alpha})$.

Finally, suppose $\left |\wedge  \right |=n$, we find  $X_{1},\cdots
,X_{n}\in \mathcal{N}_{p}$ such that $<X^{*}_{p},X_{1}>,\cdots
,<X^{*}_{p},X_{n}>$ are different from each other as well as

$\begin{vmatrix}
\int_{K}P_{\delta_{1}}(\Psi_{2}^{-1}(\overline{q_{1}(k.exp(X_{1}))}))dk &
\cdots &
\int_{K}P_{\delta_{n}}(\Psi_{2}^{-1}(\overline{q_{1}(k.exp(X_{1}))}))dk \\
\vdots  & \cdots & \vdots \\
\int_{K}P_{\delta_{1}}(\Psi_{2}^{-1}(\overline{q_{1}(k.exp(X_{n}))}))dk &
\cdots &
\int_{K}P_{\delta_{n}}(\Psi_{2}^{-1}(\overline{q_{1}(k.exp(X_{n}))}))dk
\end{vmatrix}\neq 0$

 Then

$\phi_{\gamma_{N}, \lambda_{N},\alpha_{N}}^{\nu}(exptX_{m})=\sum_{\delta
\in  \wedge}\frac{\widehat{L}_{P_{\delta}}(\omega
_{\lambda_{N},\alpha_{N}})}{dim(P_{\delta})}\int_{K}e^{ir_{N}<X_{p}^{*},X_{m}>t}P_{\delta}(\Psi_{2}^{-1}(\overline{q_{1}(k.exp(tX_{m}))}))dk$.

Similarly, $\phi_{\gamma, \lambda,\alpha}^{\nu}(exptX_{m})=\sum_{\delta
\in  \wedge}\frac{\widehat{L}_{P_{\delta}}(\omega
_{\lambda,\alpha})}{dim(P_{\delta})}\int_{K}e^{ir<X_{p}^{*},X_{m}>t}P_{\delta}(\Psi_{2}^{-1}(\overline{q_{1}(k.exp(tX_{m}))}))dk$.

Here $m=1,\cdots n$.

Since $\phi_{\gamma_{N}, \lambda_{N},\alpha_{N}}^{\nu}(exptX_{m})$
converges uniformly to $\phi_{\gamma, \lambda,\alpha}^{\nu}(exptX_{m})$
and $\gamma_{N}\rightarrow \gamma$. We obtain:

$\sum_{\delta \in  \wedge}\frac{\widehat{L}_{P_{\delta}}(\omega
_{\lambda_{N},\alpha_{N}})}{dim(P_{\delta})}\int_{K}e^{ir<X_{p}^{*},X_{m}>t}P_{\delta}(\Psi_{2}^{-1}(\overline{q_{1}(k.exp(X_{m}))}))dk$
converges uniformly to $\sum_{\delta \in
\wedge}\frac{\widehat{L}_{P_{\delta}}(\omega
_{\lambda,\alpha})}{dim(P_{\delta})}\int_{K}e^{ir<X_{p}^{*},X_{m}>t}P_{\delta}(\Psi_{2}^{-1}(\overline{q_{1}(k.exp(X_{m}))}))dk$.

Here $m=1,\cdots ,n$.

Differentiate with respect to $t$ $j$ times, where $j=1,\cdots ,n-1$ and
let $t=0$, we get

$\sum_{\delta \in  \wedge}\frac{\widehat{L}_{P_{\delta}}(\omega
_{\lambda_{N},\alpha_{N}})}{dim(P_{\delta})}\int_{K}(P_{\delta}(\Psi_{2}^{-1}(\overline{q_{1}(k.exp(X_{m}))}))dk)i^{j}r^{j}<X_{p}^{*},X_{m}>^{j}$
converges uniformly to $\sum_{\delta \in
\wedge}\frac{\widehat{L}_{P_{\delta}}(\omega
_{\lambda,\alpha})}{dim(P_{\delta})}\int_{K}(P_{\delta}(\Psi_{2}^{-1}(\overline{q_{1}(k.exp(X_{m}))}))dk)i^{j}r^{j}<X_{p}^{*},X_{m}>^{j}$.

Therefore, if $N$ large enough, for any $\varepsilon >0$, we have

$\begin{bmatrix}
 -\frac{\varepsilon}{M_{1}}     \\
\vdots  \\
-\frac{\varepsilon}{M_{1}}
\end{bmatrix}<
\begin{bmatrix}
1 & \cdots & 1\\
\vdots  & \cdots & \vdots \\
 i^{n-1}r^{n-1}<X_{p}^{*},X_{1}>^{n-1}& \cdots &
i^{n-1}r^{n-1}<X_{p}^{*},X_{n}>^{n-1}dk
\end{bmatrix}$

$\times$
$\begin{bmatrix}
\sum_{\delta \in  \wedge}\frac{\widehat{L}_{P_{\delta}}(\omega
_{\lambda_{N},\alpha_{N}})-\widehat{L}_{P_{\delta}}(\omega_{\lambda,
\alpha})}{dim(P_{\delta})}P_{\delta}(\Psi_{2}^{-1}(\overline{q_{1}(k.exp(X_{1}))}))dk\\
 \vdots\\
\sum_{\delta \in
\wedge}\frac{\widehat{L}_{P_{\delta}}(\omega_{\lambda_{N},\alpha_{N}})-\widehat{L}_{P_{\delta}}(\omega_{\lambda,
\alpha})}{dim(P_{\delta})}P_{\delta}(\Psi_{2}^{-1}(\overline{q_{1}(k.exp(X_{n}))}))dk
\end{bmatrix}<\begin{bmatrix}
 \frac{\varepsilon}{M_{1}}     \\
\vdots  \\
\frac{\varepsilon}{M_{1}}
\end{bmatrix}$

Therefore, if $M_{1}$ large enough, we have

$\left |\sum_{\delta \in  \wedge}\frac{\widehat{L}_{P_{\delta}}(\omega
_{\lambda_{N},\alpha_{N}})-\widehat{L}_{P_{\delta}}(\omega_{\lambda,
\alpha})}{dim(P_{\delta})}P_{\delta}(\Psi_{2}^{-1}(\overline{q_{1}(k.exp(X_{1}))}))dk\right
|<\frac{\varepsilon}{M_{2}}$,

where $m=1,\cdots n$.

Therefore,

$\begin{bmatrix}
 -\frac{\varepsilon}{M_{2}}     \\
\vdots  \\
-\frac{\varepsilon}{M_{2}}
\end{bmatrix}<
\begin{bmatrix}
\int_{K}P_{\delta_{1}}(\Psi_{2}^{-1}(\overline{q_{1}(k.exp(X_{1}))}))dk &
\cdots &
\int_{K}P_{\delta_{n}}(\Psi_{2}^{-1}(\overline{q_{1}(k.exp(X_{1}))}))dk \\
\vdots  & \cdots & \vdots \\
\int_{K}P_{\delta_{1}}(\Psi_{2}^{-1}(\overline{q_{1}(k.exp(X_{n}))}))dk &
\cdots &
\int_{K}P_{\delta_{n}}(\Psi_{2}^{-1}(\overline{q_{1}(k.exp(X_{n}))}))dk
\end{bmatrix}
\times \begin{bmatrix}
\frac{\widehat{L}_{P_{\delta_{1}}}(\omega
_{\lambda_{N},\alpha_{N}})-\widehat{L}_{P_{\delta_{1}}}(\omega_{\lambda,
\alpha})}{dim(P_{\delta})}\\
 \vdots\\
 \frac{\widehat{L}_{P_{\delta_{n}}}(\omega_{\lambda_{N},\alpha_{N}})-\widehat{L}_{P_{\delta_{n}}}(\omega_{\lambda,
\alpha})}{dim(P_{\delta_{n}})}
\end{bmatrix}<\begin{bmatrix}
 \frac{\varepsilon}{M_{2}}     \\
\vdots  \\
\frac{\varepsilon}{M_{2}}
\end{bmatrix}$

Therefore, if $M_{2}$ large enough, we have
$\frac{\left |\widehat{L}_{P_{\delta_{m}}}(\omega
_{\lambda_{N},\alpha_{N}})-\widehat{L}_{P_{\delta_{m}}}(\omega_{\lambda,
\alpha})\right |}{dim(P_{\delta_{m}})<\varepsilon}$

where $m=1,\cdots n$.

Thus,
$\widehat{L}_{P_{\delta_{m}}}(\omega_{\lambda_{N},\alpha_{N}})\rightarrow
\widehat{L}_{P_{\delta_{m}}}(\omega_{\lambda, \alpha})$,

for $m=1,\cdots n$.

Since $\{\gamma_{1},\cdots ,\gamma_{d}\}\subset \{P_{\delta }\mid \delta
\in \wedge\}$, this shows in particular that

$\widehat{L}_{\gamma_{j}}(\omega_{\lambda_{N},\alpha_{N}})\rightarrow
\widehat{L}_{\gamma_{j}}(\omega_{\lambda, \alpha})$,

for $j=1,\cdots d$.

Conversely, suppose
$\widehat{L}_{\gamma_{j}}(\omega_{\lambda_{N},\alpha_{N}})\rightarrow
\widehat{L}_{\gamma_{j}}(\omega_{\lambda, \alpha})$,

for $j=1,\cdots d$;
$\widehat{T}(\omega_{\lambda_{N},\alpha_{N}})\rightarrow
\widehat{T}(\omega_{\lambda, \alpha})$ and $r_{N}\rightarrow r$. It
follows that $\widehat{L}_{P}(\omega_{\lambda_{N},\alpha_{N}})\rightarrow
\widehat{L}_{P}(\omega_{\lambda, \alpha})$ for every $P\in
\mathbb{C}[V_{R}]^{K}$.

Indeed, each $P\in \mathbb{C}[V_{R}]^{K}$ is a linear combination of
monomials $\gamma^{a}$ in the fundamental invariants,

\begin{equation}\label{l-invariant elements}
\begin{split}
&\lim_{N\rightarrow \infty
}(L_{\gamma^{a}})^{\land}(\omega_{\lambda_{N},\alpha_{N}})=\lim_{N\rightarrow
\infty }(L_{\gamma}^{a}(\omega_{\lambda_{N},\alpha_{N}})+\sum_{\left \|b
\right \|<\left \|a  \right
\|}c_{a,b}\widehat{L}_{\gamma}^{b}(\omega_{\lambda_{N},\alpha_{N}})\widehat{T}(\omega_{\lambda_{N},\alpha_{N}})^{\left
\|a  \right \|-\left \|b  \right \|})\\
&=\lim_{N\rightarrow \infty
}(L_{\gamma}^{a}(\omega_{\lambda,\alpha})+\sum_{\left \|b  \right \|<\left
\|a  \right
\|}c_{a,b}\widehat{L}_{\gamma}^{b}(\omega_{\lambda,\alpha})\widehat{T}(\omega_{\lambda,\alpha})^{\left
\|a  \right \|-\left \|b  \right \|})\\
&=(L_{\gamma^{a}})^{\land}(\omega_{\lambda,\alpha}).
\end{split}
\end{equation}

Suppose that $\psi_{N}=\phi_{\gamma_{N}, \lambda_{N},\alpha_{N}}$ and
$\psi_{N}\rightarrow \psi$, where $\psi=\phi_{\gamma, \lambda,\alpha}$ is
a bounded $O(n)$-spherical function.

Since $\widehat{T}(\omega_{\lambda_{N}, \alpha_{N}})=i\lambda_{N}$
converges to $\widehat{T}(\omega_{\lambda, \alpha})=i\lambda$,
we have $\lambda_{N}\rightarrow \lambda$.

Since $\widehat{L}_{\gamma_{0}}(\omega_{\lambda_{N},\alpha_{N}})=\left
|\lambda_{N}  \right |\widehat{L}_{\gamma_{0}}(\omega_{1,\left
|\alpha_{N}\right|})=\left |\lambda_{N}  \right |(2\left |\alpha _{N}
\right |+n)$ converges to
$\widehat{L}_{\gamma_{0}}(\omega_{\lambda,\alpha})=\left |\lambda  \right
|(2\left |\alpha \right |+n)$. Also, $r_{N}:=\left |\alpha_{N} \right |$
must converges to $ \left |\alpha  \right |$. Thus, both $(\lambda_{N})$
and $(\gamma_{N})$ are bounded sequences and we choose constants
$C_{1},C_{2}$ with $\left |\lambda_{N} \right |\leq C_{1}$, $0\leq
\gamma_{N}\leq C_{2}$ for all $N$. Therefore, $\left |\lambda \right |\leq
C_{1}$,$0\leq \gamma\leq C_{2}$. Choose constant $C_{3}$ and $C_{4}$ with

$\gamma_{0}(Pr_{V}(\Psi_{2}^{-1}(\overline{q_{1}(k.n))}))=\left
|Pr_{V}(\Psi_{2}^{-1}(\overline{q_{1}(k.n))})  \right |^{2}\leq C_{3}$ for
all $n\in S$ and let $\Psi_{2}^{-1}(\overline{q_{1}(k.n))})\in S^{'} \
\forall k\in K$.

Since $S$ is a compact set, $K$ is also a compact set, $S^{'}$ is a
compact set. Here $Pr_{V}:(z,t)\rightarrow z$ and $Pr_{T}:(z,t)\rightarrow
t$.

Thus, we have

$\frac{\gamma_{0}(Pr_{V}(\Psi_{2}^{-1}(\overline{q_{1}(k.n))}))^{m}\left
|\lambda_{N}  \right |^{m}}{m!}\leq \frac{(c_{3}c_{1})^{m}}{m!}\leq
\frac{c_{4}}{2^{m}}$.

for all $m,N$ and all $Pr_{V}(\Psi_{2}^{-1}(\overline{q_{1}(k.n))})\in
S^{'}$, $\forall k\in K$.

And therefore, for $Pr_{V}(\Psi_{2}^{-1}(\overline{q_{1}(k.n))})\in
S^{'}$, $\forall k\in K$, we have

$\left
|e^{i\lambda_{N}Pr_{T}(\Psi_{2}^{-1}(\overline{q_{1}(k.n))})}\sum_{\delta
\in  \wedge, \left |\delta   \right |\geq
M}\frac{\widehat{L}_{P_{\delta}}(\omega
_{\lambda_{N},\alpha_{N}})}{dim(P_{\delta})}P_{\delta}(Pr_{V}(\Psi_{2}^{-1}(\overline{q_{1}(k.n})))
 \right |$

$\leq \sum_{m=M}^{\infty}\left |\widehat{L}_{P_{m}}(\omega
_{1,\alpha_{N}})  \right |\sum_{\left |\delta  \right
|=m}P_{\delta}(Pr_{V}(\Psi_{2}^{-1}(\overline{q_{1}(k.n})))\left
|\lambda_{N}  \right |^{m}$

$\leq
\sum_{m=M}\binom{m+n+\gamma_{N}-1}{m}\frac{1}{2^{m}m!}\gamma_{0}^{m}(Pr_{V}(\Psi_{2}^{-1}(\overline{q_{1}(k.n))}))\left|\lambda_{N}
 \right |^{m}$

$\leq
\frac{c_{4}}{2^{M}}\sum_{m=M}^{\infty}\binom{m+n+\gamma_{N}-1}{m}(\frac{1}{2})^{m}$

$\leq
\frac{c_{4}}{2^{M}}\sum_{m=0}^{\infty}\binom{m+n+\gamma_{N}-1}{m}(\frac{1}{2})^{m}$

$\leq \frac{c_{4}}{2^{M}}2^{n+\gamma_{N}}$

$\leq \frac{c_{5}}{2^{M}}$,

where $c_{5}=c_{4}.2^{n+c_{2}}$. Remember we have $\lambda_{N}\rightarrow
\lambda$, $\gamma_{N}\rightarrow \gamma$, therefore, we have
$\phi_{\gamma_{N}, \lambda_{N},\alpha_{N}}^{\nu}(n)$ converges uniformly
to $\phi_{\gamma, \lambda,\alpha}^{\nu}(n)$ on the compact set $S$.
\end{proof}

\begin{theorem}\label{equal}
Let $(\psi_{N})_{N=1}^{\infty}$ be a sequence of bounded $O(n)$-spherical
functions of type 2, and $\psi$ is a bounded $O(n)$-spherical function of
type 2. Then $\psi_{N}$ converges to $\psi$ in the topology of
$\bigtriangleup(O(n),F(n))$(i.e.uniformly on compact sets) if and only if
$r_{N}\rightarrow r$, where $\psi_{N}(n)=\int_{K}e^{ir_{N}<X_{p}^{*},k.X
>}dk$,$\psi(n)=\int_{K}e^{ir<X_{p}^{*},k.X >}dk$,$n=exp(X+A)$.
\end{theorem}
\begin{proof}
If $\psi_{N}^{\upsilon}(n)$ converges uniformly to $\psi^{\upsilon}(n)$.
Take $X\in \mathcal{N}_{p}$, such that $\int_{K}<X_{p}^{*},k.X >dk\neq 0$.

Then $\psi_{N}(exptX)=\int_{K}e^{ir_{N}<X_{p}^{*},k.X >t}dk\rightarrow
\psi(exptX)=\int_{K}e^{ir<X_{p}^{*},k.X >t}dk$ uniformly, where
$-\infty<t<\infty$.

Differentiate with respect to $t$ and let $t=0$, we obtain

$ir_{N}\int_{K}<X_{p}^{*},k.X >dk\rightarrow ir\int_{K}<X_{p}^{*},k.X >dk$.

Therefore, $r_{N}\rightarrow r$.

Conversely, if $r_{N}\rightarrow r$, its obvious that
$\psi_{N}(n)=\int_{K}e^{ir_{N}<X_{p}^{*},k.X >}dk\rightarrow
\psi(n)=\int_{K}e^{ir<X_{p}^{*},k.X >}dk$.
\end{proof}

\begin{theorem}\label{equal}
$\bigtriangleup(O(n),F(n))$ is a complete metric space.That is, if
$(\psi_{N})_{N=1}^{\infty}$ is a sequence of bounded-$O(n)$-spherical
functions that converges uniformly to $\psi$ on compact subset in $F(n)$,
then $\psi$ is a bounded $O(n)$-spherical function.
\end{theorem}
\begin{proof}
It is clear that $\psi$ is continuous, $O(n)$-invariant and bounded
and $\psi(e)=1$. Moreover, if $f,g \in L_{K}^{1}(F(n))$ have compact
support then
\begin{equation}\label{l-invariant elements}
\begin{split}
&\int_{F(n)}\psi(n)(f*g)(n)dn=\lim_{N\rightarrow
\infty}\int_{F(n)}\psi_{N}(n)(f*g)(n)dn\\
&=\lim_{N\rightarrow
\infty}\int_{F(n)}\psi_{N}(n)f(n)dn\int_{F(n)}\psi_{N}(n)g(n)dn \\
&=\int_{F(n)}\psi(n)f(n)dn\int_{F(n)}\psi(n)g(n)dn
\end{split}
\end{equation}

Thus, $f\rightarrow \int_{F(n)}\psi(n)f(n)dn$ defines a continous non-zero
algebra homomorphism $L_{K}^{1}(F(n))\rightarrow \mathbb{C}$. It follows
that $\psi\in \bigtriangleup(O(n),F(n))$.
\end{proof}

\section{The second main theorem and some other results}
The $O(N)$-Spherical transform for $f\in L_{K}^{1}(F(n))$ is the function

$\widehat{f}:\bigtriangleup(O(n),F(n))\rightarrow \mathbb{C},
\widehat{f}(\psi)=\int_{F(n)}f(n)\psi(n^{-1})dn$.

Here dn denote the haar measure for the group $F(n)$.

One has
\begin{equation}\label{l-invariant elements}
(f*g)^{\wedge}(\psi)=\widehat{f}(\psi)\widehat{g}(\psi)
\end{equation}
and
\begin{equation}\label{l-invariant elements}
(f^{*})^{\wedge}(\psi)=\overline{\widehat{f}(\psi)}
\end{equation}

for $f,g\in L_{K}^{1}(F(n))$, $\psi\in \bigtriangleup(O(n),F(n))$,
$f^{*}(n)=\overline{f(n^{-1})}$.

Let's compute equation (5.2) for example,
\begin{equation}\label{l-invariant elements}
\begin{split}
&(f^{*})^{\wedge}(\psi)=\int_{F(n)}f^{*}(n)\psi(n^{-1})dn\\
&=\int_{F(n)}\overline{f(n^{-1})}\psi(n^{-1})dn=\int_{F(n)}\overline{f(n)}\psi(n)dn
\\
&=\int_{F(n)}\overline{f(n)}\times
\overline{\psi(n^{-1})}dn=\int_{F(n)}\overline{f(n)\psi(n^{-1})}dn \\
&=\overline{\widehat{f}(\psi)}
\end{split}
\end{equation}

The compact open topology is the smallest topology makes all of the maps
$\{\widehat{f}\mid f\in L_{K}^{1}(F(n))\}$ continuous. Since
$L_{K}^{1}(F(n))\}$ is a Banach, $*$-algebra with respect to the
involution $f\rightarrow f^{*}$, it follows that $\widehat{f}$ belongs to
the space $C_{0}(\bigtriangleup(O(n),F(n)))$ of continuous functions on
$\bigtriangleup(O(n),F(n))$ that vanish at infinity. Moreover, we have
$\left \|\widehat{f}  \right \|_{\infty}\leq \left \|f  \right \|_{1}$ for
$f\in \bigtriangleup(O(n),F(n))$. This follows immediately from the fact
that  for $\psi\in \bigtriangleup(O(n),F(n))$ one has $\left |\psi(n)
\right |\leq \psi(e)=1$, since $\psi$ is positive definite.

Godement's Plancherel Theory for Gelfand pairs (cf.[9] or Section 1.6 in
[1]) ensures that there exists a unique positive Borel measure $d\mu$ on
the space $\bigtriangleup(O(n),F(n))$ for which
\begin{equation}\label{l-invariant elements}
\int_{F(n)}\left |f(n)  \right
|^{2}dn=\int_{\bigtriangleup(O(n),F(n))}\left |\widehat{f}(\psi) \right
|d\mu(\psi)
\end{equation}
for all continous functions $f\in L_{K}^{1}(F(n))\cap L_{K}^{2}(F(n))$. If
$f\in L_{K}^{1}(F(n))\cap L_{K}^{2}(F(n))$ is continuous and $\widehat{f}$
is integrable with respect to $d\mu$, then one has the Inversion Formula.
\begin{equation}\label{l-invariant elements}
f(n)=\int_{\bigtriangleup(O(n),F(n))}\widehat{f}(\psi)\overline{\psi(n^{-1})}d\mu(\psi)
\end{equation}

In particular, this formula holds when $f$ is continuous, positive
definite and $K$-invariant. Moreover, the spherical transform
$f\rightarrow \widehat{f}$ extends uniquely to an isomorphism between
$L_{K}^{2}(F(n))$ and $L^{2}(\bigtriangleup(O(n),F(n)),d\mu)$.

Let $\mathcal{L}$ be the set of $\wedge=(\lambda_{1},\cdots
,\lambda_{p^{'}})\in \mathbb{R}^{p^{'}}$ such that $\lambda_{1}>\cdots
>\lambda_{p^{'}}>0$. We define the following measure on $\mathcal{L}$:

$d\wedge=d\lambda_{1}\cdots d\lambda_{p^{'}}$ is the restricted Lebesgue
measure on $\mathcal{L}$,

$d\eta^{'}(\wedge)=\begin{cases}
c\prod_{i=1}^{p^{'}}\lambda_{i}\prod_{j<k}(\lambda_{j}^{2}-\lambda_{k}^{2})^{2}d\lambda_{1}\cdots
d\lambda_{p^{'}}      & \mbox{if }p=2p^{'} \\
c\prod_{i=1}^{p^{'}}\lambda_{i}^{3}\prod_{j<k}(\lambda_{j}^{2}-\lambda_{k}^{2})^{2}d\lambda_{1}\cdots
d\lambda_{p^{'}}, & \mbox{if }p=2p^{'}+1
\end{cases}$.

where the constant $c$ is some constant.

Over $\mathbb{R}^{+}$, we define the measure $\tau$ given as the Lebesgue
measure if $p=2p^{'}+1$, and the Dirac measure in 0, if $p=2p^{'}$.

$c(p)=\begin{cases}
(2\pi)^{-\frac{p(p-1)}{2}+p^{'}}& \mbox{if }p=2p^{'} \\
2(2\pi)^{-\frac{p(p-1)}{2}+p^{'}-1}, & \mbox{if }p=2p^{'}+1
\end{cases}$.

\begin{theorem}\label{equal}
$m^{*}$ is the radial Plancherel measure for $(N_{p},O_{p})$,i.e. for a
$K$-invariant function $\psi\in L^{2}(N)$, we have
\begin{equation}\label{l-invariant elements}
\left \|\psi  \right \|^{2}_{L^{2}(N)}=\int \left
|<\psi,\phi^{r,\wedge,l}>\right |^{2}dm^{*}(r,\wedge,l).
\end{equation}
\end{theorem}
Note that $m^{*}$ is given as the tensor product of $\eta^{'}$ on
$\mathcal{L}$, and the counting measure $\sum$ on $\mathbb{N}^{p^{'}}$,
and the measure $\tau$ on $\mathbb{R}^{+}$, up to the normalizing constant
$c(p)$. According to the definition of $d\eta^{'}$, the second type of the
bounded $O(n)$-spherical functions has no compact on the above formula.

\begin{theorem}\label{equal}
The bounded $O(n)$-spherical functions of "type 1" are dense in the space
$\bigtriangleup(O(n),F(n))$.
\end{theorem}
\begin{proof}
Take a point $r\in R^{+}$, and suppose that $\phi^{r,0}$ is not in the
closure of $\{\phi^{r,\left |\lambda  \right |,\alpha}\mid r\in R^{+},\left
|\lambda  \right |\in R^{+}, \alpha\in \wedge\}$.
$\bigtriangleup(O(n),F(n))$ is metrizable, hence it is completely regular.
So we can find a continuous function
$J:\bigtriangleup(O(n),F(n))\rightarrow \mathbb{R}$ with
$J(\phi^{r,0})=1$, $J(\phi^{r,\left |\lambda  \right |,\alpha})=0$ for all
$r\in R^{+},\left |\lambda  \right |\in R^{+}, \alpha\in \wedge$. We can
assume that $J$ has compact support.
\\The equation (5.2) ensures that $L_{K}^{1}(F(n))$ is a symmetric banach
$*$-algebra. It follows that $\{\widehat{f}\mid f\in L_{K}^{1}(F(n))$ is
dense in $(C_{0}(\bigtriangleup(O(n),F(n))),\left \|\cdot \right
\|_{\infty})$. (See for example, $\S$ 14 in chapter 3 of [12].) Thus we
can find a sequence $(j_{N})$ in $L_{K}^{1}(F(n))$ with
$\widehat{j_{N}}\rightarrow J$ uniformly on $\bigtriangleup(O(n),F(n))$.
We can assume that each $j_{N}$ is continous and compactly supported.
Moreover, since $J$ is real-valued, we can assume that $j_{N}^{*}=j_{N}$.

Similar to the proof of Proposition 3 in [10] shows that one can find an
approximate identity $(a_{s})_{s>0}$ in $L_{K}^{1}(F(n))$ with
$\widehat{a_{s}}$ compactly supported in $\bigtriangleup(O(n),F(n))$ for
all $s>0$. For $s$ sufficiently small, one has
$\widehat{a_{s}}(\phi^{r,0})>\frac{3}{4}$. moreover, for each $s$ one sees
that $(a_{s}*j_{N})^{\wedge}=\widehat{a_{s}}\widehat{j_{N}}$ converges
uniformly to $\widehat{a_{s}}J$ as $N\rightarrow \infty$. Thus we can
choose $s_{0}$ sufficiently small and $N_{0}$ sufficiently large that
$g=a_{s_{0}}*j_{N_{0}}$ satisfies
$\widehat{a_{s}}(\phi^{r,0})>\frac{1}{2}$ and $\left
|\widehat{a_{s}}(\phi^{r,\left |\lambda  \right |,\alpha})  \right
|<\frac{1}{4}$
for all $r\in R^{+},\left |\lambda  \right |\in R^{+}, \alpha\in \wedge$

Note that $g$ is continuous, integrable, square-integrable and $g^{*}=g$.

Dixmier's functional Calculus (cf.[11]) ensures that "sufficiently smooth
functions operate on $L_{K}^{1}(F(n))$." Thus if $\zeta:
\mathbb{R}\rightarrow \mathbb{R}$ is sufficiently smooth with $\zeta$ and
its derivatives integrable and $\zeta(0)=0$, then there is a function
$f:=\zeta \{g\}\in L_{K}^{1}(F(n))\cap L_{K}^{2}(F(n))\cap C(F(n))$ with
the property that $\widehat{f}=\zeta{\widehat{g}}$. We choose a $\zeta$
with $\zeta(t)=1$ for $t>\frac{1}{2}$ and  $\zeta(t)=0$ for
$t<\frac{1}{4}$. Then $F=\widehat{f}=\widehat{\zeta \{g\}}$ satisfies
$F(\phi^{r,0})=1$ and $F(\phi^{r,\left |\lambda  \right |,\alpha})=0$ for
all $r\in R^{+},\left |\lambda  \right |\in R^{+}, \alpha\in \wedge$.

Now theorem 5.6 shows
\begin{equation}\label{l-invariant elements}
\begin{split}
&\left \|\psi  \right \|^{2}_{L^{2}(N)}=\int \left
|<\psi,\phi^{r,\wedge,l}>\right |^{2}dm^{*}(r,\wedge,l)\\
&=\int \left |\int_{F(n)}\psi(x)\overline{\phi^{r,\wedge,l}(x)}  \right
|^{2}dxdm^{*}(r,\wedge,l)
=\int \left |\int_{F(n)}\psi(x)\phi^{r,\wedge,l}(x^{-1})  \right
|^{2}dxdm^{*}(r,\wedge,l)\\
&=\int \left |\widehat{\psi}(\phi^{r,\wedge,l})\right |^{2}dm^{*}(r,\wedge,l)
\end{split}
\end{equation}

From this and equation (5.4), we obtain $F=0$ a.e. on
$\bigtriangleup(O(n),F(n))$. In particular, $F$ is integrable on
$\bigtriangleup(O(n),F(n))$ and we apply formula (5.5) to conclude
$f\equiv 0$ on $F(n)$. This implies that $F=\widehat{f}$ is identically
zero on $\bigtriangleup(O(n),F(n))$, which contradicts $F(\phi^{r,0})=1$.
\end{proof}

We define the Fourier transform $\mathcal{F}_{\mathcal{N}}(f)$:
$\mathcal{N}\rightarrow \mathbb{C}$ of $f\in L^{1}(\mathcal{N})$ by
\begin{equation}\label{l-invariant elements}
{F}_{\mathcal{N}}(f)(w)=\int_{\mathcal{N}}f(z)e^{-ir<z,w>}dz
\end{equation}
where dz denotes Euclidean measure on $\mathcal{N}_{\mathbb{R}}$ and
$\mathcal{N}$ is the lie algebra of $N_{p}$. With this normalization, one
has $\left \|\mathcal{F}_{\mathcal{N}}(f)  \right \|_{2}=(2\pi)^{n}\left
\|f  \right \|_{2}$. This is because

\begin{equation}\label{l-invariant elements}
\begin{split}
&\left \|\mathcal{F}_{\mathcal{N}}(f)  \right
\|_{2}=\int_{\mathcal{N}}\int_{\mathcal{N}}\int_{\mathcal{N}}f(z)e^{-ir<z,w>}\overline{f(z^{'})}e^{ir<z^{'},w>}dzdz^{'}dw\\
&=\int_{\mathcal{N}}\int_{\mathcal{N}}f(z)\overline{f(z^{'})}\int_{\mathcal{N}}e^{ir<z^{'}-z,w>}dzdz^{'}dw\\
&=(2\pi)^{n}\int_{\mathcal{N}}f(z)\overline{f(z)}dz=(2\pi)^{n}\left \|f
\right \|_{2}
\end{split}
\end{equation}

There are several properties about the Fourier transform, they are as
follows:

(a)$\mathcal{F}_{\mathcal{N}}(c_{1}f_{1}+c_{2}f_{2})=c_{1}\mathcal{F}_{\mathcal{N}}(f_{1})+c_{2}\mathcal{F}_{\mathcal{N}}(f_{2})$,

for $c_{1}，c_{2}\in \mathbb{C}$ and $f_{1}，f_{2}\in L^{1}(\mathcal{N})$.

(b)$\left |\mathcal{F}_{\mathcal{N}}(f)(w)  \right |\leq
\int_{\mathcal{N}}\left |f(z)  \right |dz=\left \|f  \right \|_{1}$ for
all $f\in L^{1}(\mathcal{N})$.

(c)$\mathcal{F}_{\mathcal{N}}(f*g)(w)=\mathcal{F}_{\mathcal{N}}(f)(w)\mathcal{F}_{\mathcal{N}}(g)(w)$
for all $f,g\in  L^{1}(\mathcal{N})$.

(d)$\mathcal{F}_{\mathcal{N}}(\tilde{f})=\overline{\mathcal{F}_{\mathcal{N}}(f)}$
for all $f\in L^{1}(\mathcal{N})$.

(e)Let $(L_{z^{'}}f)(z)=f(z-z^{'})$, then
$(L_{z^{'}}f)(w)=e^{-ir<z^{'},w>}\mathcal{F}_{\mathcal{N}}(f)(w)$.

Conversely,
$\mathcal{F}_{\mathcal{N}}(e^{ir<z^{'},z>}f(z))(w)=L_{Z^{'}}\mathcal{F}_{\mathcal{N}}(f)(w)$
for all $f\in L^{1}(\mathcal{N})$.
\begin{proof}
(a)\begin{equation}\label{l-invariant elements}
\begin{split}
&\mathcal{F}_{\mathcal{N}}(c_{1}f_{1}+c_{2}f_{2})=\int_{\mathcal{N}}(c_{1}f_{1}+c_{2}f_{2})(z)e^{-ir<z,w>}dz\\
&=c_{1}\int_{\mathcal{N}}f_{1}(z)e^{-ir<z,w>}dz+c_{2}\int_{\mathcal{N}}f_{2}(z)e^{-ir<z,w>}dz\\
&=c_{1}\mathcal{F}_{\mathcal{N}}(f_{1})+c_{2}\mathcal{F}_{\mathcal{N}}(f_{2})
\end{split}
\end{equation}

(b)$\left |\mathcal{F}_{\mathcal{N}}(f)(w)  \right |=\left
|\int_{\mathcal{N}}f(z)e^{-ir<z,w>}dz  \right |$

$\leq \int_{\mathcal{N}}\left |f(z)  \right |dz=\left \|f  \right \|_{1}$.

(c)
\begin{equation}\label{l-invariant elements}
\begin{split}
&\mathcal{F}_{\mathcal{N}}(f*g)(w)\\
&=\int_{\mathcal{N}}\int_{\mathcal{N}}f(z-x)g(x)e^{-ir<z,w>}dzdx\\
&=\int_{\mathcal{N}}\int_{\mathcal{N}}f(y)g(x)e^{-ir<x+y,w>}dxdy\\
&=\int_{\mathcal{N}}\int_{\mathcal{N}}f(y)g(x)e^{-ir<y,w>}e^{-ir<x,w>}dxdy\\
&=\int_{\mathcal{N}}f(y)e^{-ir<y,w>}dy\int_{\mathcal{N}}g(x)e^{-ir<x,w>}dx\\
&=\mathcal{F}_{\mathcal{N}}(f)(w)\mathcal{F}_{\mathcal{N}}(g)(w)
\end{split}
\end{equation}

(d)
\begin{equation}\label{l-invariant elements}
\begin{split}
&\mathcal{F}_{\mathcal{N}}(\tilde{f})(\omega)=\int_{\mathcal{N}}\tilde{f}(z)e^{-ir<z,w>}dz\\
&=\int_{\mathcal{N}}\overline{f(-z)}\times \overline{e^{-ir<-z,w>}}dz\\
&=\overline{\int_{\mathcal{N}}f(z)e^{-ir<z,w>}dz}\\
&=\overline{\mathcal{F}_{\mathcal{N}}(f)}(\omega)
\end{split}
\end{equation}

(e)
\begin{equation}\label{l-invariant elements}
\begin{split}
&(L_{z^{'}}f)(w)\\
&=\int_{\mathcal{N}}(L_{z^{'}}f)(z)e^{-ir<z,w>}dz\\
&=\int_{\mathcal{N}}f(z-z^{'})e^{-ir<z,w>}dz\\
&=\int_{\mathcal{N}}f(z)e^{-ir<z+z^{'},w>}dz\\
&=e^{-ir<z^{'},w>}\mathcal{F}_{\mathcal{N}}(f)(w)
\end{split}
\end{equation}

conversely,
\begin{equation}\label{l-invariant elements}
\begin{split}
&\mathcal{F}_{\mathcal{N}}(e^{ir<z^{'},z>}f(z))(w)\\
&=\int_{\mathcal{N}}e^{ir<z^{'},z>}f(z)e^{-ir<z,w>}dz\\
&=\int_{\mathcal{N}}f(z)e^{-ir<z,w-z^{'}>}dz\\
&=\mathcal{F}_{\mathcal{N}}(f)(w-z^{'})\\
&=L_{Z^{'}}\mathcal{F}_{\mathcal{N}}(f)(w)
\end{split}
\end{equation}

Finally, since we know the two different types bounded $O(n)$-spherical
functions, we compute the spherical transform of them respectively.
Remember the two different types bounded $O(n)$-spherical functions are
as follows:
For $n=exp(X+A)\in N$.

Type 1:$\phi^{r,\wedge,l}(n)=\int_{K}e^{ir<X_{p}^{*},k.X
>}\omega_{\wedge,l}(\Psi_{2}^{-1}(\overline{q_{1}}(k.n)))dk$.

Type 2:$\phi^{\upsilon}(n)=\int_{K}e^{ir<X_{p}^{*},k.X >}dk$.

For any $f\in L_{K}^{1}(F(n))$  and ”type 1” bounded $O(n)$-spherical functions,we have
\begin{equation}\label{l-invariant elements}
\begin{split}
&\hat{f}(\phi^{r,\wedge,l})=\int_{N}f(n)\int_{K}e^{-ir<X_{p}^{*},k.X
>}\omega_{\wedge,l}(\Psi_{2}^{-1}(\overline{q_{1}}(k.n^{-1})))dkdn\\
&=\int_{N}f(k.n)\int_{K}e^{-ir<X_{p}^{*},k.X
>}\omega_{\wedge,l}(\Psi_{2}^{-1}(\overline{q_{1}}(k.n^{-1})))dkdn\\
&=\int_{K}\int_{N}f(n)\int_{K}e^{-ir<X_{p}^{*},X
>}\omega_{\wedge,l}(\Psi_{2}^{-1}(\overline{q_{1}}(n^{-1})))dndk\\
&=\int_{N}f(n)e^{-ir<X_{p}^{*},X
>}\omega_{\wedge,l}(\Psi_{2}^{-1}(\overline{q_{1}}(n^{-1})))dn
\end{split}
\end{equation}

For any $f\in L_{K}^{1}(F(n))$  and ”type 2” bounded $O(n)$-spherical functions,we have
\begin{equation}\label{l-invariant elements}
\begin{split}
&\hat{f}(\phi^{\upsilon})=\int_{N}f(n)\int_{K}e^{-ir<X_{p}^{*},k.X
>}dkdn\\
&=\int_{N}f(k.n)\int_{K}e^{-ir<X_{p}^{*},k.X
>}dkdn\\
&=\int_{K}\int_{N}f(n)\int_{K}e^{-ir<X_{p}^{*},X
>}dndk\\
&=\int_{N}f(n)e^{-ir<X_{p}^{*},X
>}dn
\end{split}
\end{equation}

\end{proof}

\end{document}